\documentclass[a4paper, draft, 12pt]{amsproc}
\usepackage{amssymb}
\usepackage{amsmath,amssymb,ulem,color}

\usepackage[hyphens]{url} 
\usepackage[dvips]{graphicx} 
\usepackage{cmdtrack} 
\usepackage{mathrsfs}                       

\theoremstyle{plain}
 \newtheorem{theorem}{Theorem}[section]

\theoremstyle{Definition}
 \newtheorem{exm}{Example}[section]
 
 \newtheorem{dfn}{Definition}[section]
\theoremstyle{remark}
 
 \newtheorem{conj}{Conjecture}[section]
 \numberwithin{equation}{section}

\renewcommand{\geq}{\geqslant}
\renewcommand{\setminus}{\smallsetminus}
\title[Fast Escaping Set of Transcendental Semigroup]{Fast Escaping Set of Transcendental Semigroup}



\author[B. H. Subedi]{\bfseries  Bishnu Hari Subedi}

\address{ 
Central Department of Mathematics \\ 
Institute of Science and Technology   \\ 
Tribhuvan University   \\ 
Kirtipur, Kathmandu\\
Nepal}
\email{subedi.abs@gmail.com / subedi\_bh@cdmathtu.edu.np }

\author[A. Singh]{Ajaya Singh}
\address{Central Department of Mathematics, Institute of Science and Technology, Tribhuvan University, Kirtipur, Kathmandu, Nepal }
\email{singh.ajaya1@gmail.com / singh\_a@cdmathtu.edu.np} 
\thanks{This research work of first author is supported by PhD faculty fellowship of University Grants Commission, Nepal. } 

\begin{document}

{\begin{flushleft}\baselineskip9pt\scriptsize
MANUSCRIPT
\end{flushleft}}
\vspace{18mm} \setcounter{page}{1} \thispagestyle{empty}

\begin{abstract}
In this paper, we study fast escaping set of transcendental semigroup. We discuss the structure and properties of fast escaping set of transcendental semigroup. We also see how far the classical theory of fast escaping set of transcendental entire function applies to general settings of transcendental semigroups and what new phenomena can occur. 
\end{abstract}

\maketitle

\section{Introduction}

The principal aim of this paper is to study fast escaping set not for iterates of single transcendental entire maps of  complex plane $ \mathbb{C} $  but for the composite of the family $\mathscr{F}  $ of such maps. Let $ \mathscr{F} $ be a space of transcendental entire maps on $ \mathbb{C} $. For any map $ f \in \mathscr{F} $,  $ \mathbb{C} $ is naturally partitioned  into two subsets: the set of normality and its complement. The \textit{set of normality} or \textit{Fatou set} $ F(f) $ of $ f \in \mathscr{F} $ is the largest open set on which the iterates $ f^{n} = f \circ f\circ \ldots \circ f$ (n-fold composition of $ f $ with itself, $ n \in \mathbb{N} $) is a normal family.  The complement of Fatou set in $ \mathbb{C} $ is the \textit{Julia set} $ J(f) $. A maximally connected subset of the Fatou set $ F(f) $ is a \textit{Fatou component}.  The \textit{escaping set} of any $ f \in \mathscr{F} $ is the set
$$
I(f) = \{z\in \mathbb{C}:f^n(z)\rightarrow \infty \textrm{ as } n\rightarrow \infty \}
$$
and any point $ z \in I(S) $ is called \textit{escaping point}. For transcendental entire function $f$, the escaping set $I(f)$ was first studied by A. Eremenko \cite{ere}. He showed that 
$I(f)\not= \emptyset$;  $ J(f) =\partial I(f) $;
 $I(f)\cap J(f)\not = \emptyset$; and 
$\overline{I(f)}$ has no bounded component. 
 By motivating from this last statement, he posed a conjecture: 
\begin{conj} \label{ec}
  Every component of $ I(f) $ unbounded.
 \end{conj}
This conjecture is considered  as an important open problem of transcendental dynamics and nowadays it is famous as \textit{Eremenko's conjecture}\index{Eremenko's conjecture}. 
The Eremenko's conjecture \ref{ec} in general case has been proved by using the fast escaping set $A(f)$, which consists of points whose iterates tends to infinity as fast as possible. This set is a subset of escaping set and it was introduced  first time by Bergweiler and Hinkkanen \cite{ber} 
and defined in the following form by Rippon and Stallard \cite{rip4}. For a transcendental entire function $f$, the \textit{fast escaping set}\index{fast escaping ! set} is a set of the form:
            \[A(f) = \{z\in \mathbb{C}: \exists L\in \mathbb{N} \ \text{such that}\ \ |f^{n+L}(z)|\geq M^n(R)\;\ \text{for}\;\  n\in \mathbb{N}\}\]
where $M(r) = \max_{|z|=r}|f(z)|,  r>0$ and $M^n(r)$ denotes nth iteration of $M(r)$ with respect to $r$.  $R>0$ can be taken any value such that $M(r)>r$ for $r\geq{R}$. 

 Recall that the set $ C(f) = \{z\in \mathbb{C} : f^{\prime}(z) = 0 \}$ is the set of \textit{critical points} of the transcendental entire function $ f $ and  the set $CV(f) = \{w\in \mathbb{C}: w = f(z)\;\ \text{such that}\;\ f^{\prime}(z) = 0\} $  is called the set of \textit{critical values}.  The set 
$AV(f)$ consisting of all  $w\in \mathbb{C}$ such that there exists a curve (asymptotic path) $\Gamma:[0, \infty) \to \mathbb{C}$ so that $\Gamma(t)\to\infty$ and $f(\Gamma(t))\to w$ as $t\to\infty$ is called the set of \textit{asymptotic values} of $ f $ and the set
$SV(f) =  \overline{(CV(f)\cup AV(f))}$
is called the \textit{singular values} of $ f $.  
If $SV(f)$ has only finitely many elements, then $f$ is said to be of \textit{finite type}. If $SV(f)$ is a bounded set, then $f$ is said to be of \textit{bounded type}.                                                                                                                             
The sets
$$\mathscr{S} = \{f:  f\;\  \textrm{is of finite type}\} 
\;\;  \text{and}\; \;                                                                                                                                                                                                                                                                                                                                                                                                                                                                                                                                                                                                                                                                                                                                                                                                                                                                                                                                                                                                                                                                                                                                                                                                                                                                                                                                                                                                                                                                                                                                                                                                                                                                                                                                                                                                                                                                                                                                                                                                                                                                                                                                                                                                                                                                                                                                                                                                                                                                                                                                                                                                                                                                                                                                                                                                                                                                                                                                                                                                                                                                                                                     
\mathscr{B} = \{f: f\;\  \textrm{is of bounded type}\}
$$
are respectively called \textit{Speiser class} and \textit{Eremenko-Lyubich class}. 

The main concern of such a transcendental iteration theory is to describe the nature of the components of Fatou set and the structure and properties of the Julia set, escaping set and fast escaping set. We use monograph: \textit{dynamics of transcendental entire functions} \cite{hou} and book: \textit{holomorphic dynamics} \cite{mor} for basic facts concerning the Fatou set, Julia set and escaping set of a transcendental entire function. We use  \cite{rip1, rip4, six} for facts and results concerning the fast escaping set of a transcendental entire function.    

Our particular interest is to study of the dynamics of the families that are semigroups  generated by the class $\mathscr{F}$ of transcendental entire maps.  For a collection $\mathscr{F} = \{f_{\alpha}\}_{\alpha \in \Delta} $ of such maps, let $S =\langle f_{\alpha} \rangle $ 
be a \textit{transcendental semigroup} generated by them.  The index set $ \Delta $ to which $ \alpha $  belongs is allowed to be infinite in general unless otherwise stated. 
Here, each $f \in S$ is a transcendental entire function and $S$ is closed under functional composition. Thus, $f \in S$ is constructed through the composition of finite number of functions $f_{\alpha_k},\;  (k=1, 2, 3,\ldots, m) $. That is, $f =f_{\alpha_1}\circ f_{\alpha_2}\circ f_{\alpha_3}\circ \cdots\circ f_{\alpha_m}$. 
A semigroup generated by finitely many transcendental entire functions $f_{i}, (i = 1, 2, \ldots,  n) $  is called \textit{finitely generated  transcendental semigroup}. We write $S= \langle f_{1},f_{2},\ldots,f_{n} \rangle$.
 If $S$ is generated by only one transcendental entire function $f$, then $S$ is \textit{cyclic transcendental semigroup}. We write $S = \langle f\rangle$. In this case, each $g \in S$ can be written as $g = f^n$, where $f^n$ is the nth iterates of $f$ with itself. Note that in our study of  semigroup dynamics, we say $S = \langle f\rangle$  a \textit{trivial semigroup}. The transcendental semigroup $S$ is \textit{abelian} if  $f_i\circ f_j =f_j\circ f_i$  for all generators $f_{i}$ and $f_{j}$  of $ S $. The transcendental semigroup $ S $ is \textit{bounded type (or finite type)} if each of its generators $ f_{i} $ is bounded type (or finite type).
  
The family $\mathscr{F}$  of complex analytic maps forms a \textit{normal family} in a domain $ D $ if given any composition sequence $ (f_{\alpha}) $ generated by the member of  $ \mathscr{F} $,  there is a subsequence $( f_{\alpha_{k}}) $ which is uniformly convergent or divergent on all compact subsets of $D$. If there is a neighborhood $ U $ of the point $ z\in\mathbb{C} $ such that $\mathscr{F} $ is normal family in $U$, then we say $ \mathscr{F} $ is normal at $ z $. If  $\mathscr{F}$ is a family of members from the semigroup $ S $, then we simply say that $ S $ is normal in the neighborhood of $ z $ or $ S $ is normal at $ z $.

Let  $ f $ be a transcendental entire map. We say that  $ f $ \textit{iteratively divergent} at $ z \in \mathbb{C} $ if $  f^n(z)\rightarrow \infty \; \textrm{as} \; n \rightarrow \infty$.  Semigroup $ S $ is \textit{iteratively divergent} at $ z $ if $f^n(z)\rightarrow \infty \; \textrm{as} \; n \rightarrow \infty$ for all $ f \in S $. Otherwise, a function $ f  $ and semigroup $ S $  are said to be \textit{iteratively bounded} at $ z $. Note that in a semigroup $S = \langle f_{\alpha} \rangle$ if $ f \in S $ then for all $ m \in \mathbb{N}, f^{m} \in S$. So, $f^{m} =f_{\alpha_1}\circ f_{\alpha_2}\circ f_{\alpha_3}\circ \cdots\circ f_{\alpha_p}$ for some $ p\in \mathbb{N} $. In this sense,  any $ f \in S $ is iteratively divergent at any $ z \in \mathbb{C} $ always means that there is a sequence $ (g_{n})_{n \in \mathbb{N}} $ in $ S $ representing  $ g_{1} = f, g_{2}= f^{2}, \ldots , g_{n} = f^{n}, \ldots $ such that $ g_{n}(z) \to \infty $ as $ n \to \infty $ or there a sequence in $ S $ which contains $ (g_{n})_{n \in \mathbb{N}} $ as a subsequence such that $ g_{n}(z) \to \infty $ as $ n \to \infty $. More generally, semigroup $ S $ is iteratively divergent at any $ z \in \mathbb{C} $ always means that every $ f \in S $ is iteratively divergent at $ z $. That is, every sequence in $ S $  has a subsequence which diverges to infinity at $ z $.   

Based  the Fatou-Julia-Eremenko theory,  the Fatou set, Julia set and escaping set in the settings of transcendental  semigroup are defined as follows.
\begin{dfn}[\textbf{Fatou set, Julia set and escaping set}]\label{2ab} 
\textit{Fatou set} of the holomorphic semigroup $S$ is defined by
  $$
  F (S) = \{z \in \mathbb{C}: S\;\ \textrm{is normal in a neighborhood of}\;\ z\}
  $$
and the \textit{Julia set} $J(S) $ of $S$ is the compliment of $ F(S) $. If $ S $ is a transcendental semigroup, the \textit{escaping set} of $S$ is defined by 
$$
I(S)  = \{z \in \mathbb{C}: S \;  \text{is iteratively divergent at} \;z \}.
$$
We call each point of the set $  I(S) $ by \textit{escaping point}.        
\end{dfn} 
It is obvious that $F(S)$ is the largest open subset of $\mathbb{C}$ on which the family $\mathscr{F} $ in $S$ (or semigroup $ S $ itself) is normal. Hence its compliment $J(S)$ is a smallest closed set for any  transcendental semigroup $S$. Whereas the escaping set $ I(S) $ is neither an open nor a closed set (if it is non-empty) for any transcendental semigroup $S$. Any maximally connected subset $ U $ of the Fatou set $ F(S) $ is called a \textit{Fatou component}.  
If $S = \langle f\rangle$, then $F(S), J(S)$ and $I(S)$ are respectively the Fatou set, Julia set and escaping set in classical complex dynamics. In this situation we simply write: $F(f), J(f)$ and $I(f)$. 

In our study, classical transcendental dynamics  refers the iteration theory of single transcendental map and transcendental semigroup dynamics refers the dynamical theory generated by a set of transcendental entire maps. In transcendental semigroup dynamics, algebraic structure of semigroup naturally attached to the dynamics and hence the situation is largely complicated. The principal aim  of this paper is to see how far classical transcendental dynamics applies to transcendental semigroup dynamics and what new phenomena appear in transcendental semigroup settings. 
  
The fundamental contrast between classical transcendental dynamics and semigroup dynamics appears by different algebraic structure of corresponding semigroups. In fact, non-trivial semigroup  need not be, and most often will not be abelian. However, trivial semigroup is cyclic and therefore abelian. As we discussed before, classical complex dynamics is a dynamical study of trivial (cyclic)  semigroup whereas semigroup dynamics is a dynamical study of non-trivial transcendental semigroup.

The main motivation of this paper comes from seminal work of Hinkkanen and Martin \cite{hin} on the dynamics of rational semigroup and the extension study of K. K. Poon \cite{poo} to the dynamics of transcendental semigroup. Both of them naturally generalized classical complex dynamics to the dynamics of the sequence of different functions by means of composition. Another motivation of studying escaping set of transcendental semigroup comes from the work of Dinesh Kumar and Sanjay Kumar \cite{kum2, kum1} where they defined escaping set and discussed how far escaping set of classical transcendental dynamics can be generalized to semigroup dynamics. In parallel, we also studied structure and propperties of Fatou set, Julia set and escaping under semigroup dynamics in \cite{sub1, sub2, sub3, sub8, sub7, sub4, sub5, sub6}. From these attempts, we again more motivate to study fast escaping set of transcendental semigroup. 

In this paper, we introduce fast escaping set in transcendental semigroup settings which is a main concern of our study. Note that the fast escaping set $A(f)$ in classical transcendental dynamics introduced first time by Bergweiler and Hinkkanen \cite{ber} and studied in more depth by Rippon and Stallard \cite{rip4}. Indeed,  it is a set consisting of points whose iterates tends to infinity as fast as possible and now plays a key role in classical transcendental dynamics. The set $A(f)$ has some properties exactly similar to those of $I(f)$.  For example $A(f)\not = \emptyset$, in fact, it is infinite set ({\cite[Lemma 2]{ber}}),  $\partial A(f) = J(f)$ ({\cite[Lemma 3]{ber}}) and( {\cite [Theorem 5.1 (b)]{rip4}}), $A(f) \cap J(f)\not = \emptyset $ ({\cite[Lemma 3]{ber}}) and({\cite[Theorem 5.1(a)]{rip4}}), The set $ A(f) $  is completely invariant under $ f $ ({\cite [Theorem 2.2 (a)]{rip4}}) and the set $ A(f) $ does not depend on the choice of $ R> 0 $   ({\cite [Theorem 2.2 (b)]{rip4}}).  However,  in {\cite [Theorem 1]{rip1}} and {\cite [Theorem 1.1]{rip4}},  it is shown that all components of $A(f)$ are unbounded and since $ A(f)\subset I(f) $ provides a partial answer to Eremenko's conjecture in normal form. If $U$ is a Fatou component of $f$ in $A(f)$, then its boundary is also in $A(f)$, that is,  $\overline{U}\subset A(f)$ ({\cite [Theorem 2]{rip1}}) and ({\cite[Theorem 1.2]{rip4}}). These are  much stronger properties of $A(f)$  than those of escaping set $I(f)$. 

\section{Fast Escaping Set of Transcendental Semigroup}

There is no equivalent formulation of fast escaping set in semigroup settings. We have started to define fast escaping set and try to formulate some other related terms and results. Note that it is our seminal work on the study of fast escaping set in transcendental semigroup settings. 

 Let  $S$ be a transcendental semigroup. Let us define a set of the form
\begin{equation}\label{fe1}
A_R (S) =  \{z\in \mathbb{C}: |f^n (z)|\geq M^n(R) \;\ \textrm{for all}\;\  f \in S \;\  \textrm{and}\;\ n\in \mathbb{N} \}
\end{equation}
where $M(r) =   \max_{|z|=r}|f(z)|$, with $r\geq R > 0$ and  $M^n(r)$ denotes the nth iterates of $M(r)$ with itself. $ R > 0 $  can be taken any value such that $ M(r) > r $ for $ r\geq R$. If $r$ is sufficiently large then $M^n(r)\rightarrow \infty $ as $n\rightarrow \infty $. The set $A_R (S)  $ is well defined in semigroup $ S $ because for any $ f \in S $,  $ f^{n} \in S $ for all $ n \in \mathbb{N} $. From the condition $ |f^n (z)|\geq M^n(R) \;\ \textrm{for all}\;\  f \in S \;\  \textrm{and}\;\ n\in \mathbb{N} $ of the set $A_R (S) $, we can also say that a point $ z \in \mathbb{C} $ is in $ A_R (S) $  if every sequence $ (g_{n})_{n \in \mathbb{N}} $  in $ S $ has a subsequence  $ (g_{n_{k}})_{n_{k} \in \mathbb{N}} $ 
which increases without bound at least as fast as the n-iterates of the maximum modulus of each $ g_{n_{k}} $: $M(r) =   \max_{|z|=r}|g_{n_{k}}(z)|$. Where $ g_{n_{k}}= f^{n} $ such that $ |f^n (z)|\geq M^n(R) $. 

\begin{dfn}[\textbf{Fast escaping set}] \label{fe2} 
The fast escaping set $A(S)$\index{fast escaping ! set} of a transcendental semigroup $ S $ consists the set $A_R(S)$ and all its pre-images. In other words, fast escaping set is the set of the form
\begin{equation}\label{fe3}
A(S) = \bigcup_{n\geq0}f^{-n}(A_R(S))
\end{equation}
where $f^{-n}(A_{R}(S)) = \{z \in \mathbb{C}: f^{n}(z) \in A_{R}(S) \} $ for all $ f\in S $ and $ n \in \mathbb{N} $.
\end{dfn}
We can do certain stratification of fast escaping set which helps to make it more visible  and provides a significant new understanding of the structure and properties of this set.  We can write fast escaping set as a countable union of all its labels  as we define below. 
\begin{dfn}[\textbf{L th label of fast escaping set}]\label{fe9}
Let $ A(S) $ be a fast escaping set of transcendental semigroup $ S $. For $ L\in \mathbb{Z} $, then the set of the form
\begin{multline}\label{fe6}
A_R^L (S)   = \{z\in \mathbb{C}: |f^n (z)|\geq M^{n+L}(R)\;  \textrm{for all}\;\  f \in S, \\ n\in \mathbb{N}\;  \text{and}\;\ n + L \geq 0 \}
\end{multline}
is called \textit{Lth level}\index{Lth level of fast escaping set} of fast escaping set $ A(S) $. 
\end{dfn}

Note that the set $ A_{R}(S) $  defined above in  \ref{fe1} is the  \textit{0th level}  of fast escaping set $ A(S) $. Since $ M^{n+1}(R) > M^{n}(R) $ for all $ n \geq 0 $, so from \ref{fe6}, we get following chain of relation 

\begin{multline}\label{fe7}
\ldots \subset A_R^L (S) \subset A_R^{L-1} (S)\subset  \ldots \subset A_R^1 (S)\subset A_R (S)\subset  \\ A_R^{-1} (S)\subset A_R^{-2}(S)\subset \ldots\subset  A_{R}^{-(L-1)}(S)\subset A_{R}^{-L}(S) \subset \ldots
\end{multline}
From \ref{fe3} and \ref{fe7}, the fast escaping set  can also be written as an expanding union of its labels. 
\begin{equation}\label{fe8}
A(S) = \bigcup_{L\in \mathbb{N}}A_R^ {-L}(S)
\end{equation} 
Again from the definition \ref{fe2}, that is,  from \ref{fe3}, if any $ z_{0} \in A(S) $, then $ z_{0} \in f^{-n}(A_{R}(S)) $ for some $ n \geq 0 $. It gives $ f^{n}(z_{0}) \in A_{R}(S) $ for all $ f \in S $. From \ref{fe2}, there is $ L \in \mathbb{N} $ such that $ |f^{L}(f^{n}(z_{0}))| =|f^{n + L}(z_{0})| \geq M^{n}(R) $. With this clause, the fast escaping set of transcendental semigroup $ S $ can now be written as
 
\begin{multline}\label{fe5}
A(S) = \{z\in \mathbb{C}: \text{there exists}\;\ L\in \mathbb{N}\;\ \text{such that}\\  |f^{n+ L} (z)|\geq M^{n}(R) \;  \text{for all}\;\ f\in S, \;\text{and}\;\ n\in \mathbb{N}\}.
\end{multline}
We can use any one of the form \ref{fe3} or \ref{fe8} or \ref{fe5} as a definition of fast escaping set in our forth coming study.  Note that by the definition (\ref{fe3} or \ref{fe8} or \ref{fe5}), fast escaping set $ A(S) $ of any transcendental semigroup $ S $ is a subset of escaping set $ I(S) $. Since from {\cite[Theorem 1.2 (3)]{sub1}}, we can say that $ I(S) $ may be empty. For any transcendental semigroup $ S $, if $ I(S) =\emptyset $, then we must have $ A(S)= \emptyset $.  It is not known whether there is a transcendental semigroup $ S $ such that $ I(S)\neq \emptyset $ but $ A(S)= \emptyset $. Note that in classical transcendental dynamics, both of these sets are non-empty. 
\begin{exm}
Let $ S $ be a transcendental semigroup generated by functions $ f(z) =e^{z} $ and $ g(z) =e^{-z} $. Since $ h = g \circ f^{n}  \in S$ is iteratively bounded at any $ z \in \mathbb{C} $. So,  $ I(S) =\emptyset $ and $ A(S) =\emptyset $. 
\end{exm}
Like escaping set $ I(S) $, fast escaping set $ A(S) $ is also neither open nor closed set if it is non-empty. Similar to the result {\cite[Theorem 1.2 (3)]{sub1}} of escaping set, the following result is also clear from the definition of fast escaping set.
\begin{theorem}\label{fe11}
Let $ S $ be a transcendental semigroup. Then
$A(S) \subset A(f)$ for all $f \in S$  and hence  $A(S)\subset \bigcap_{f\in S}A(f)$. 
\end{theorem}

We have mentioned several results and examples of transcendental semigroups in \cite{sub1, sub2, sub3, sub7, sub4,  sub5} where escaping set is non-empty. One of important particular result in this regards is a condition for which escaping set of a transcendental semigroup is same as escaping set of its each element. In such case, the fact would be obvious from classical transcendental dynamics that the fast escaping set is also non-empty. In our fourth coming study, we always talk with such a semigroup whose fast escaping set is non-empty. 

\section{Elementary Properties of Fast Escaping Set}
In this section, we check how far basic properties of fast escaping set of classical transcendental dynamics can be generalize to fast escaping set of transcendental semigroup dynamics. In \cite{sub2}, we examined the  contrast between classical complex dynamics and semigroup dynamics  in the invariant features of Fatou set, Julia set and escaping set. In this paper, we see the same type contrast in fast escaping set.  Note that in classical transcendental dynamics, the fast escaping set is completely invariant. 

\begin{dfn}[\textbf{Forward, backward and completely invariant set}]
For a given semigroup $S$, a set $U\subset \mathbb{C}$ is said to be $ S $-\textit{forward invariant}\index{forward ! invariant set} if $f(U)\subset U$ for all $f\in S$. It is said to be $ S $-\textit{backward invariant}\index{backward ! invariant set} if $f^{-1}(U) = \{z \in \mathbb{C}: f(z) \in U \}\subset U$ for all $ f\in S $ and it is called $ S $-\textit{completely invariant}\index{completely invariant ! set} if it is both S-forward and S-backward invariant.
\end{dfn} 

We prove the following elementary results that are important regarding the structure of fast escaping set $ A(S) $. These results may also have  more chances of leading further results concerning the properties and structure of $ A(S) $. Indeed, it shows certain connection and contrast between classical transcendental dynamics and transcendental semigroup dynamics and it is also a nice generalization of classical transcendental dynamics to semigroup dynamics. 
\begin{theorem}{\label{2}}
Let $ S $ is a transcendental semigroup such that $ A(S) \neq \emptyset $. Then the following are hold.
\begin{enumerate}
\item $ A(S) $ is S-forward invariant.
\item $ A(S) $ is independent of $ R $.
\item $ J(S) = \partial A(S) $.
\item $ J(S) \subset \overline{A(S)} $. 
\item $ A(S) \cap J(S) \neq \emptyset $.

\end{enumerate}
\end{theorem}

\begin{proof}

(1). From the definition \ref{fe9} (that is,  from equation \ref{fe6}), we can write 
$$
A_R^L (S) \subset \{z\in \mathbb{C}: |z| \geq M^{L}(R), \; L \geq 0 \}.
$$ 
So for any $ z_{0} \in  A_R^L (S)$, 
$$f(z_{0}) \in\{z\in \mathbb{C}: |f(z)| \geq M^{L+ 1}(R), \; L \geq 0 \} =  A_R^{L+1} (S)
$$ 
for all $ f \in S $ and $ n \in \mathbb{N} $. 
This shows that 
$ f(A_R^L (S)) \subset A_R^{L+1} (S)$ for all $ f \in S $.  
However from   relation \ref{fe7}, 
$
A_R^{L+1} (S) \subset A_R^L (S). 
$
Hence, we have 
$
f(A_R^L (S)) \subset A_R^L (S). 
$
This fact together with equation \ref{fe8}, we can say that $ A(S) $ is S-forward invariant.

(2) Choose $ R_{0} > R $, then from \ref{fe7}, we have
$A^{L}_{R_{0}} (S) \subset A_R^L (S)$ for all $ L \in \mathbb{Z} $ and so 
$$
 \bigcup_{L \in \mathbb{N}}A^{-L}_{R_{0}} (S) \subset\bigcup_{L \in \mathbb{N}}A^{-L}_{R} (S)  
$$
Since there is $ m \in \mathbb{N} $ such that $ M^{m}(R) > R_{0} $ and so
$$
\bigcup_{L \in \mathbb{N}}A^{-L}_{R} (S) \subset \bigcup_{L \in \mathbb{N}}A^{m-L}_{R} (S) = \bigcup_{L \in \mathbb{N}}A^{-L}_{M^{m}(R)} (S) \subset \bigcup_{L \in \mathbb{N}}A^{-L}_{R_{0}} (S) 
$$
From above two inequality, we have 
$$
\bigcup_{L \in \mathbb{N}}A^{-L}_{R_{0}} (S) = \bigcup_{L \in \mathbb{N}}A^{-L}_{R} (S)  = A(S)
$$
This proves $ A(S) $ is independent of $ R $.

(3) We prove this statement by showing $ (A(S))^{0}\subset F(S) $ and $ (A(S))^{e}\subset F(S) $ where $ (A(S))^{0} $ and $ (A(S))^{e} $ are respectively interior and exterior of $ A(S) $. Since $A(S)$ is S-forward invariant, so $ f^{n}(A(S)) \subset A(S) $ for all $ f \in S $ and $ n \in \mathbb{N} $.  
 Suppose $z \in  (A(S))^{0} $, then there is a neighborhood $ V $ of $ z $ such that $ z \in V \subset A(S) $. Since $ A(S) $ contains no periodic points, so $ |f^{n+ L} (z)|\geq M^{n}(R) \;  \text{for all}\;\ f\in S, \;\text{and}\;\ n\in \mathbb{N} $ and hence  $ (f^{n})_{n \in \mathbb{N}} $ is normal on $ V $ by Montel's theorem. Thus $ z \in F(S) $.  This proves $ (A(S))^{0}\subset F(S) $. 
 
By the theorem 3.2.3 of \cite{bea}, the closure and complement of $ A(S) $ are also forward invariant. So 
from $ f^{n}(A(S)) \subset A(S) $, we can write
$$f^{n}(\mathbb{C} -\overline{A(S)})\subset \mathbb{C} -\overline{A(S)}. 
$$ 
for all $ n \in \mathbb{N} $. Since $\mathbb{C} -\overline{A(S)} = (A(S))^{e} $.
By the assumption of non-empty $ A(S) $, $ \overline{A(S)} $ is also a non-empty closed set. By the definition, $ F(S) $ is a largest open set on which $ S $ is normal family, so we must   $$\mathbb{C} -\overline{A(S)} = (A(S))^{e}\subset F(S) $$. 

(4) The proof follows from (3). 

(5) By the theorem \ref{fe11}, $A(S) \subset A(f)$ for all $f \in S$. A Fatou component $ U \subset F(S) $ is also a component of $ F(f) $ for each $ f \in S $. \\
Case (i): If $ U $ is multiply connected component of $ F(S) $, then by {\cite[Theorem 2 (a)]{rip1}} $ \overline{U} \subset A(f) $ for all $ f \in S $. Again by the above same theorem \ref{fe11}, $\overline{U} \subset A(S) $. This shows that $ \partial U \subset A(S) $. Since $ \partial U \subset J(f) $ for all $ f \in S $. By {\cite[Theorem 4.2]{poo}}, we write $ \partial U \subset J(S) $.  This proves $ A(S) \cap J(S) \neq \emptyset $.\\
case(ii): If $ U $ is simply connected component of $ F(S) $ that meets $ A(S) $,  then by {\cite[Theorem 1.2 (b)]{rip4}} $ \overline{U} \subset A(f) $ for all $ f \in S $. So, as in case (i), $ \overline{U} \subset A(S) $. By {\cite[Corollary 4.6]{rip4}}, if  $ F(S) $ has only simply connected components, then $ \partial{A_{R}^{L}} \subset J(S) $ where $ \partial{A_{R}^{L}(S)} $ is L-th label of $ F(S) $. From the equation \ref{fe8}, we conclude that $ A(S) \cap J(S) \neq \emptyset $.
\end{proof}

There are many classes of functions from which we get $ I(f) \subset J(f) $ and for such functions, we must have $ A(f) \subset J(f) $. Dinesh Kumar and Sanjay Kumar {\cite[Theorem 4.5]{kum2}} prove that $ I(S) \subset J(S) $ if transcendental semigroup $ S $ is of finite or bounded type. We prove the following  similar result. 
\begin{theorem}
Let $ S $ be a bounded or finite type transcendental semigroup. Then $ A(S) \subset J(S) $ and $ J(S) = \overline{A(S)} $. 
\end{theorem} 
\begin{proof}
For each $ f \in S $, Eremenko and Lyubich \cite{ere1} proved that $ I(f) \subset J(f) $.  K. K. Poon {\cite[Theorem 4.2]{poo}}  proved that $ J(S) =\overline {\bigcup_{f \in S} J(f)}$. So for any $ f \in S $, we have $ A(S) \subset A(f) \subset J(f) \subset J(S) $. The second part follows from $ A(S) \subset J(S) $ together with above theorem \ref{2} (4).  
\end{proof}

There are many functions in the class $ \mathscr{B} $, the escaping set $ I(f) $ consists of uncountable family of curves tending to infinity. For example, function  $ \lambda \sin z + \gamma $ with $ \lambda, \gamma \in \mathbb{C} $ belongs to the class $ \mathscr{S} \subset \mathscr{B} $ and its escaping set is an uncountable union of curves tending to infinity, the so-called Cantor bouquet.  For the function $ f(z) = \lambda e^{z}, 0< \lambda < 1/e $, the Fatou set is completely invariant attracting basin and Julia set is a Cantor bouquet consisting of uncountably many disjoint simple curves,  each of which has finite end point and other endpoint is $ \infty $. The escaping set of such a function consists of open curves (without endpoints) together with some of their end points. Note that for such a function, each point in the escaping set can be connected to $ \infty $ by a curve in the escaping set.  For such functions, every point in such a curve  belongs to fast escaping set except possibly a finite endpoint. More generally, let $f$ be a finite composition of functions of finite order in the class $\mathscr{B}$ and let $z_0 \in  I(f)$. Then  $z_0$ can be connected to $ \infty $ by a simple curve $\Gamma \subset I(f)$ such that $\Gamma \setminus \{z_0\} \subset A(f)$ (see for instance {\cite[Theorem 1.2]{rem}}). 

There may a chance of similar result in semigroup dynamics if semigroup $ S $ is generated by the transcendental functions  of finite order in the class $\mathscr{B}$. If so, then every $ f \in S $ is a finite composition of the functions of finite order in the class $\mathscr{B}$ and for each of such function $ f $,  $ A(f) $ consists of curves $ \Gamma \setminus \{z_0\} $ with exception of some of the end points. Since $ A(S) \subset A(f) $ for each $ f \in S $, then $ A(S) $ may consist of curves $ \Gamma \setminus \{z_0\} $ with exception of some of the end points.

\section{On the L-th Labels of $ A(S) $}
In this section, we more concentrate on L-th label $ A_{R}^{L}(S) $ of fast escaping set $ A(S) $. Since fast escaping set can be written as expanding union of L-th labels, so we hope that structure and properties of each L-th label may determine structure and properties of fast escaping set. Again, we will also see  contrast between fast escaping set and and its label if there are. The following result is a contrast. That is, analogous to classical transcendental dynamics \cite{rip4}, unlike the set $ A(S) $, each of its label is a closed set. 

\begin{theorem}{\label{1}}
Let $ L\in\mathbb{Z} $, then for a transcendental semigroup $ S $ such that $ A(S) \neq \emptyset $. Then the set $A_R^L (S)$ is closed and unbounded for each $ L\in \mathbb{Z} $ if it is non-empty. 
\end{theorem}
\begin{proof}
From the definition \ref{fe9}, we can write $ A_{R}^{L}(S) \subset A_{R}^{L}(f)   $ for all $ f \in S $. This implies that $A_{R}^{L}(S) \subset \bigcap_{f \in S} A_{R}^{L}(f)  $. Since for each $ L \in \mathbb{Z} $, $ A_{R}^{L}(f)  $ is a closed and unbounded set and also by {\cite[Theorem 1.1]{rip4}} each component of $ A_{R}^{L}(f)  $ is closed and unbounded for all $ f\in S $. So,  $\bigcap_{f \in S} A_{R}^{L}(f) $ is also a closed and unbounded set and each of its component is closed and unbounded as well. Since $ A_{R}^{L}(S) $ is a component of $\bigcap_{f \in S} A_{R}^{L}(f) $, so it must be closed and unbounded.
\end{proof}

On the light of this theorem \ref{1} and equation \ref{fe8}, we have obtained a new structure of fast escaping set $ A(S) $, a countable union of closed and unbounded sets $ A_{R}^{L}(S) $. This result also provides a solution of Eremenko's conjecture \ref{ec}
in transcendental semigroup settings. This  generalizes the result of classical transcendental dynamics to transcendental semigroup dynamics. 

Labels of fast escaping set $ A(S) $ can be used to show if $ U $ is a Fatou component in $ A(S) $, then boundary of $ U  $ is also in $ A(S) $. There are variety of results on simply connected and multiply connected Fatou components. Each of the Fatou  component of transcendental semigroup is either a stable (periodic) or unstable (wandering (non- periodic)) domain as defined below.

 \begin{dfn}[\textbf{Stabilizer,  wandering component and  stable domains}]\label{1g}
For a transcendental  semigroup $ S $, let $ U $ be a component of the  Fatou set $ F(S) $ and $ U_{f} $ be a component of Fatou set containing $ f(U) $ for some $ f\in S $.  The set of the form 
$$S_{U} = \{f\in S : U_{f} = U\}  $$
is called \textit{stabilizer} of $ U $ on $ S $. If $ S_{U} $ is non-empty,  we say that a component $ U $ satisfying  $U_{f} = U  $ is called \textit{stable basin} for  $ S $. The component $ U $ of $ F(S) $ is called wandering if the set $ \{U_{f}: f \in S \} $ contains infinitely many elements. That is, $ U $ is a wandering domain if there is sequence of elements $ \{f_{i}\} $ of $ S $ such that $ U_{f_{i}}  \neq U_{f_{j}}$ for $ i \neq j $. Furthermore, the component $ U $ of $ F(S) $ is called strictly wandering if $U_{f} = U_{g} $ implies $ f =g $. A stable basin $ U $ of a transcendental semigroup $ S $ is
\begin{enumerate}
\item  \textit{attracting} if it is a subdomain of attracting basin of each $ f\in S_{U} $
\item  \textit{supper attracting} if it is a subdomain of supper attracting basin of each $ f\in S_{U} $
\item  \textit{parabolic} if it is a subdomain of parabolic basin of each $ f\in S_{U} $
\item  \textit{Siegel} if it is a subdomain of Siegel disk of each $ f\in S_{U} $
\item  \textit{Baker} if it is a subdomain of Baker domain  of each $ f\in S_{U} $

\end{enumerate}
\end{dfn}
 Note that  stabilizer $ S_{U} $ is a a subsemigroup of 
$ S $ ({\cite[Lemma 2.2]{sub4}}). Also, in classical case, a stable basin is one of above type.  For any Fatou component $ U $, we prove the following result which is analogous to {\cite[Theorem 1.2]{rip4}} of classical transcendental dynamics. 
\begin{theorem}\label{fe12}
Let $ U $ be a Fatou component of transcendental semigroup $ S $ that meets $ A_{R}^{L}(S) $, where $ R > 0 $ be such that $ M(r, f) > r $ for $ r \geq R $ for all $ f \in S $ and $ L \in \mathbb{N} $. Then 
\begin{enumerate}
\item $ \overline{U} \subset  A_{R}^{L-1}(S) $;
\item if $ U $ is simply connected, then $ \overline{U} \subset  A_{R}^{L}(S) $
\end{enumerate}
\end{theorem}

\begin{proof}
Since $ U \cap A_{R}^{L}(S) \neq \emptyset $. The fact $A_{R}^{L}(S)  \subset A_{R}^{L}(f)   $ for all $ f \in S $ implies that $ U \cap A_{R}^{L}(f) \neq \emptyset $ for all $ f \in S $. So, from the theorem {\cite[Theorem 1.2 (a)]{rip4}}, we always have  $ \overline{U} \subset  A_{R}^{L-1}(f) $ for all $ f \in S $. So $ \overline{U} \subset  A_{R}^{L-1}(S) $. The second part also follows similarly by using {\cite[Theorem 1.2 (b)]{rip4}}. 
\end{proof}
By part (2) of above theorem \ref{fe2}, we can  conclude that $ \overline{U} \subset A_{R}^{L}(S) $ for all simply connected component $ U $ of $ F(S) $.  So, if all components of $ F(S) $ are simply  connected, then we must $ \partial A_{R}^{L}(S) \subset J(S) $ and hence interior of $A_{R}^{L}(S)$ is contained in $ F(S) $. This theorem also generalizes the result of classical transcendental dynamics to transcendental semigroup dynamics. That is, whatever Fatou component (simply or multiply connected)  $ U $ of $ F(S) $ intersecting $ A(S) $, there is always $ \overline{U} \subset A(S) $. Again, another question may raise. Such a Fatou component $ U $ is periodic or wandering? Note that in classical transcendental dynamics, such a Fatou component is always wandering ({\cite[Corollary 4.2]{rip4}}). For transcendental semigroup dynamics, such a Fatou component is again wandering domain. For, if $ U \cap A(S) \neq \emptyset $, then $ U \cap A(f) \neq \emptyset $ for all $ f \in S $. In this case $ U $ is wandering domain of each $ f \in S $, so it is  wandering domain of $ S $ as well. 

Whatever discussion we have done above was about a Fatou component intersecting the fast escaping set $ A(S) $ of a transcendental semigroup $ S $. Are there any Fatou components that are obviously known to lie in $ A(S) $? In classical transcendental dynamics, its answer is yes (see for instance {\cite[Theorem 4.4]{rip4}} and {\cite[Theorem 2]{rip1}}). Indeed, in such case, the Fatou component that obviously lie in $ A(S) $ is a  (closure of) multiply connected wandering domain.  Bergweiler constructed an example of transcendental entire function $ f $ for which $ A(f) $ contains simply connected wandering domain({\cite[Theorem 2]{ber2}}). This wandering domain is simply connected bounded one which lie in between multiply connected wandering domains and this one is only known example of non-multiply  connected Fatou component that lie in $ A(f) $. The generalization of above discussion to semigroup dynamics of course possible. For example, if $ U $ is a multiply connected wandering domain of $ F(S) $, then it also multiply connected wandering domains of every $ f \in S $. In this case, $ \overline{U} \subset A(f) $  for all $ f \in S $ ({\cite[Theorem 4.4]{rip4}}). Hence $ \overline{U} \subset A(S) $.

\end{document}